\documentclass[11pt,final]{amsart}

\usepackage{amssymb,amsmath,latexsym}
\usepackage{amsthm}
\usepackage{amsfonts}
\usepackage{enumerate}
\usepackage{booktabs}
\usepackage{mathrsfs}

\usepackage{manfnt}

\usepackage{hyperref}
\usepackage{xcolor}
\hypersetup{
    colorlinks, linkcolor={red!80!black},
    citecolor={blue!80!black}, urlcolor={blue}
}
\usepackage[color]{showkeys}
\definecolor{refkey}{gray}{.75}
\definecolor{labelkey}{gray}{.75}

\usepackage{graphicx,pgf}
\usepackage{float}

\usepackage{caption}
\usepackage{subcaption}

\usepackage{fancyhdr}

\newtheorem{theorem}{Theorem}[section]

\newtheorem{proposition}{Proposition}[section]
\newtheorem{lemma}{Lemma}[section]
\newtheorem{remark}{Remark}[section]

\newcommand{\aiminset}[1]{\left\{#1 \right\}}
\newcommand{\aiminabs}[1]{\lvert #1 \rvert}
\newcommand{\aiminnorm}[1]{\| #1 \|}
\newcommand{\aimininner}[2]{\langle #1, #2 \rangle}

\numberwithin{equation}{section}
\numberwithin{figure}{section}

\let\oldtocsection=\tocsection
\let\oldtocsubsection=\tocsubsection
\let\oldtocsubsubsection=\tocsubsubsection
\renewcommand{\tocsection}[2]{\hspace{0em}\oldtocsection{#1}{#2}}
\renewcommand{\tocsubsection}[2]{\hspace{1.8em}\oldtocsubsection{#1}{#2}}
\renewcommand{\tocsubsubsection}[2]{\hspace{4.5em}\oldtocsubsubsection{#1}{#2}}

\thanks{This work was partially supported by the National Science Foundation under the grant NSF DMS-1206438 and by the Research Fund of Indiana University.}

\begin{document}
\title{The 2D Euler-Boussinesq equations in planar polygonal domains with Yudovich's type data}
\author{Aimin Huang}
\address{The Institute for Scientific Computing and Applied Mathematics, Indiana University, 831 East Third Street, Rawles Hall, Bloomington, Indiana 47405, U.S.A.}
\email{huangepn@gmail.com}
\date{\today}

\begin{abstract}
We address the well-posedness of the 2D (Euler)-Boussinesq equations with zero viscosity and positive diffusivity in the polygonal-like domains with Yudovich's type data, which gives a positive answer to part of the questions raised in \cite{LPZ11}. Our analysis on the the polygonal-like domains essentially relies on the recent elliptic regularity results for such domains proved in \cite{BDT13, DT13}.
\end{abstract}

\maketitle

\setcounter{tocdepth}{2}
\tableofcontents
\addtocontents{toc}{~\hfill\textbf{Page}\par}

\section{Introduction}
Motivated by the well-posedness results for the 2D Euler equations in non-smooth domains in \cite{BDT13, DT13} and the questions about the Boussinesq system over non-smooth domains raised in \cite[Section~4]{LPZ11}, in this article, we aim to address the global well-posedness of the 2D Euler-Boussinesq equations in a non-smooth domain $\Omega\subset \mathbb R^2$.
The 2D Euler-Boussinesq equations describing the evolution of mass and heat flow of the inviscid incompressible fluid read in the non-dimensional form:
\begin{equation}\begin{cases}\label{eq1.1.1}
	\partial_t\boldsymbol u + \boldsymbol u\cdot \nabla \boldsymbol u + \nabla\pi = T \boldsymbol e_2,\quad \boldsymbol e_2=(0,1),\\
	{\mathrm{div}\,}\boldsymbol u=0,\\
	\partial_t T  - \kappa \Delta T + \boldsymbol u\cdot \nabla T =0,
\end{cases}\end{equation}
where $(x,y)\in\Omega$, $t\in(0, t_1)$, $\boldsymbol u=(u_1,u_2)$ and $T$ denote the velocity field and the temperature of the fluid respectively, and $\pi$ stands for the pressure and $\kappa> 0$ is the thermal diffusivity.
We associate with \eqref{eq1.1.1} the following initial and boundary conditions:
\begin{equation}\begin{cases}\label{eq1.1.2}
	\boldsymbol u(0,x,y)=\boldsymbol u_0(x,y),\quad\quad T(0,x,y)=T_0(x,y),\\
	\boldsymbol u(t,x,y)\cdot \boldsymbol n=0,\qquad (x,y)\in\partial\Omega,\\
	T=\eta,\qquad (x,y)\in\partial\Omega,
\end{cases}\end{equation}
where $\boldsymbol n$ is the outward unit normal vector to $\partial\Omega$ and $\boldsymbol u_0, T_0$ and $\eta$ are the given initial and boundary data. We also denote by $\boldsymbol\tau$ the unit tangent vector to $\partial\Omega$.

The general 2D Boussinesq system with full viscosity $\nu$ and diffusivity $\kappa$ reads
\begin{equation*}\begin{cases}
	\partial_t\boldsymbol u -\nu\Delta\boldsymbol u+ \boldsymbol u\cdot \nabla \boldsymbol u + \nabla\pi = T \boldsymbol e_2,\quad \boldsymbol e_2=(0,1),\\
	{\mathrm{div}\,}\boldsymbol u=0,\\
	\partial_t T  - \kappa \Delta T + \boldsymbol u\cdot \nabla T=0.
\end{cases}\end{equation*}
From the mathematical point of view, the global well-posedness, global regularity as well as the existence of the global attractor of the Boussinesq system have been widely studied, see for example \cite{CD80, FMT87, Guo89, ES94, CN97, MZ97, Wang05, CLR06, Wang07, Xu09, KTW11, CW12}.
Recently, there are many works devoted to the study of the 2D Boussinesq system with partial viscosity, see for example \cite{HL05, Cha06, HK07, DP09, HK09,  HKR11} in the whole space $\mathbb R^2$ and \cite{Zha10, LPZ11, HKZ13} in bounded smooth domains. There are also many works which considered the case when only the horizontal viscosity and vertical viscosity is present, see for example \cite{ACW11, DP11, CW13, MZ13}.
However, the global regularity for the 2D Boussinesq system when $\nu=\kappa=0$ is still an outstanding open problem and to the best of our knowledge, the well-posedness issue regarding the 2D Euler-Boussinesq system \eqref{eq1.1.1} in non-smooth domains has not yet been addressed in the literatures, which is the goal of this article. In some realistic applications, the variation of the fluid viscosity and thermal diffusivity with the temperature may not be disregarded (see for example \cite{LB96} and references therein) and there are many works on this direction, too, see for example \cite{LB96, LB99, SZ13, LPZ13, Hua14} where the existence of weak solutions, global regularity, and existence of global attractor have been studied.

It is well known that the standard 2D Euler equations is globally well-posedness if the initial data satisfies Yudovich's type condition, see \cite{Yud63, Yud95, Kel11}. Roughly speaking, if the initial vorticity is bounded or unbounded but with small growth rate of the $L^p$-norm, then the 2D Euler equations exist a global unique solution and recently this result has been extended to non-smooth domains in \cite{BDT13, DT13}.
Note that the global well-posedness for the 2D Euler-Boussinesq system has been studied in \cite{DP09} with Yudovich's type data for the whole space $\mathbb R^2$ and also studied in \cite{Zha10} with $H^3$-regular data for bounded smooth domains.
Here, we would like to establish the global well-posedness result for the 2D Euler-Boussinesq system in non-smooth domains with Yudovich's type data, which generalizes the results in \cite{DP09, Zha10} and gives a definite answer to part of questions asked in \cite{LPZ11}. We also remark that the author in \cite{Zha10} only studied the case when the boundary data is constant while here we will consider  arbitrary boundary data for the 2D Euler-Boussinesq system.

In this article, we are interested in the polygonal-like (non-smooth) domains with maximum aperture $\max \alpha_j\leq \pi/2$ due to the elliptic regularity result are only available for such domains (see \eqref{eq1.5.5} below). Here, a domain $\Omega\subset\mathbb R^2$ is said to be a polygonal-like domain if it is a bounded simply connected open set and the boundary $\partial\Omega$ is enclosed by piecewise $\mathcal C^{2}$ planar curves, with finitely many points $\{O_j\}_{j=1}^N$ of discontinuity for the tangent vector, and such that, in some neighborhood of each point $O_j$, $\Omega$ coincides with the cone of vertex $O_j$ and aperture $\alpha_j\in (0, 2\pi)$.

In order to deal with the non-homogeneous boundary conditions on $\partial\Omega$ for the temperature $T$, we use a classical lifting result (see e.g. \cite[Theorem~1.5.1.3]{Gri85}).
We assume that the boundary data $\eta$ is independent of time $t$ for the sake of simplicity.
\begin{lemma}\label{lem2.3.1}
	Let $\Omega\subset\mathbb R^2$ be a polygonal-like domain and assume that $\eta\in H^{3/2}(\partial\Omega)$. Then there exists a function $S\in H^2(\Omega)$ satisfying that
	\[
	S=\eta,\quad \frac{\partial S}{\partial \boldsymbol n}=0,\quad\text{ on }\partial\Omega,
	\]
	and the estimate
	\[
	\aiminnorm{ S }_{ H^2(\Omega)} \leq c_1\aiminnorm{ \eta }_{H^{3/2}(\partial\Omega) },
	\]
	holds for some constant $c_1>0$.
\end{lemma}

Following a traditional approach (see for example \cite{Tem97, Wang07}), we recast the 2D Euler-Boussinesq system in terms of the perturbative variable (perturbation away from the stationary state $(0, S)$); namely we set
\begin{equation*}
	(\boldsymbol u,\theta)=(\boldsymbol u, T-S).
\end{equation*}
In the perturbative variables, the 2D Euler-Boussinesq system \eqref{eq1.1.1}-\eqref{eq1.1.2} reads
\begin{equation}\begin{cases}\label{eq1.1.3}
	\partial_t\boldsymbol u  + \boldsymbol u\cdot \nabla \boldsymbol u + \nabla\pi = \theta \boldsymbol e_2 + S\boldsymbol e_2,\quad \boldsymbol e_2=(0,1),\\
	{\mathrm{div}\,}\boldsymbol u=0,\\
	\partial_t \theta - \Delta\theta + \boldsymbol u\cdot\nabla\theta +\boldsymbol u\cdot\nabla S =  \Delta S, \\
\end{cases}\end{equation}
with the initial and boundary conditions
\begin{equation}\begin{cases}\label{eq1.1.4}
	\boldsymbol u(0,x,y)=\boldsymbol u_0(x,y),\quad\quad \theta(0,x,y)=\theta_0(x,y):=T_0(x,y)-S(x,y),\\
	\boldsymbol u(t,x,y)\cdot \boldsymbol n=0,\quad\quad \theta(t,x,y)=0,\qquad (x,y)\in\partial\Omega.
\end{cases}\end{equation}
Note that we have set the diffusivity $\kappa=1$ in \eqref{eq1.1.3} for simplicity.

The rest of the article is organized as follows. At the end of this introduction, we introduce the notion of \emph{suitable weak solution}  for the 2D Euler-Boussinesq system \eqref{eq1.1.3}-\eqref{eq1.1.4} and state our main result. We prove the existence of the \emph{suitable weak solution}  by the vanishing viscosity method, which was used by Bardos in \cite{Bar72} to study the 2D Euler equations. In Section \ref{sec2}, we collect necessary tools for the analysis of the Boussinesq system in the polygonal-like domains. Section \ref{sec3} is devoted to prove the uniform estimates for the approximated solutions constructed by the vanishing viscosity method. We finally in Section \ref{sec4} to prove the main result Theorem \ref{thm2.1} below, that is the existence of the \emph{suitable weak solution} and also the regularity and uniqueness of the 2D Euler-Boussinesq system \eqref{eq1.1.3}-\eqref{eq1.1.4}. The proof of the uniqueness follows Yudovich's energy method and relies on the endpoint $L^\infty(\Omega)\to L^{\gamma_\mathrm{exp}}(\Omega)$ regularity result for the solution to the Dirichlet problem in the polygonal-like domains.
In Appendix \ref{sec-app-euler}, we recast the standard $L^p$-estimate for the vorticity of the 2D Euler equations.

\subsection{Definition of the \emph{suitable weak solution} and the main result}
In order to set up the framework of how to study the 2D Euler-Boussinesq system \eqref{eq1.1.3}-\eqref{eq1.1.4}, we recall the classical space
\[
V=\aiminset{\boldsymbol u\in H^1(\Omega)\,:\,{\mathrm{div}\,}\boldsymbol u=0,\quad \boldsymbol u\cdot \boldsymbol n=0\text{ on }\partial\Omega},
\]
and we say that a couple $(\boldsymbol u, \theta)$ satisfying
\begin{equation}\begin{split}
	\boldsymbol u\in L^\infty(0, t_1; V),\qquad \partial_t\boldsymbol u \in L^2(0, t_1; L^{3/2}(\Omega));\\
	\theta\in \mathcal C([0, t_1]; H^1_0(\Omega))\cap L^2(0,t_1; H^2(\Omega)),\qquad \partial_t\theta\in L^2(0, t_1; L^2(\Omega)),
\end{split}\end{equation}
is a \emph{suitable weak solution}  of the problem \eqref{eq1.1.3}-\eqref{eq1.1.4} if
\begin{equation*}\begin{split}
	-\int_0^{t_1}\aimininner{\boldsymbol u(t)}{\tilde{\boldsymbol u}}_{L^2}\psi'(t){\mathrm{d}t}
	+\int_0^{t_1}\aimininner{\boldsymbol u(t)\cdot\nabla \boldsymbol u(t)}{\tilde{\boldsymbol u}}_{L^2}\psi(t){\mathrm{d}t}\\
	=\aimininner{\boldsymbol u_0}{\tilde{\boldsymbol u}}\psi(0)+\int_0^{t_1}\aimininner{\theta\boldsymbol e_2 + S\boldsymbol e_2}{\tilde{\boldsymbol u}}_{L^2}\psi(t){\mathrm{d}t},
\end{split}\end{equation*}
for all $\tilde{\boldsymbol u}\in L^3_{\boldsymbol\tau}(\Omega)$ and $\psi\in \mathcal C^1([0, t_1])$ with $\psi(t_1)=0$, and
\begin{equation*}\begin{split}
	-\int_0^{t_1}\aimininner{\theta}{\tilde\theta}_{L^2}\varphi'(t){\mathrm{d}t} + \int_0^{t_1}\aimininner{\nabla\theta}{\nabla\tilde\theta}_{L^2}\varphi(t){\mathrm{d}t} + \int_0^{t_1}\aimininner{\boldsymbol u\cdot\nabla(\theta+S)}{\tilde\theta}_{L^2}\varphi(t){\mathrm{d}t}\\
	=\aimininner{\theta_0}{\tilde\theta}\varphi(0) + \int_0^{t_1}\aimininner{\Delta S}{\tilde\theta}_{L^2}\varphi(t){\mathrm{d}t},
\end{split}\end{equation*}
for all $\tilde\theta\in L^2(\Omega)$ and $\varphi\in \mathcal C^1([0, t_1])$ with $\varphi(t_1)=0$. For the meaning of the notation $L^p_{\boldsymbol\tau}(\Omega)$ ($1< p< \infty$), see Section \ref{sec2}.

The existence of a global weak solution when the boundary data $\eta$ and hence $S$ are constants is obtained using the fixed point theory in \cite{Zha10}. It seems that the fixed point argument could not be adapted to the case for the arbitrary boundary data. Here, we are going to utilize the vanishing viscosity method to prove the existence of a global \emph{suitable weak} solution and furthermore prove the global well-posedness of the 2D Euler-Boussinesq \eqref{eq1.1.3}-\eqref{eq1.1.4} with Yudovich's type data. We now state the main result of this article, with the proof presented in the Sections~\ref{sec3}-\ref{sec4}.
\begin{theorem}\label{thm2.1}
	Let $\Omega$ be a polygonal-like domain (piecewise $\mathcal C^2$-boundary) with maximum aperture $\alpha_j\leq \pi/2$ and let there be given that $S\in H^2(\Omega)$, $\boldsymbol u_0\in V$, $\theta_0\in H_0^1(\Omega)$, and $t_1>0$. Then there exists a global \emph{suitable weak} solution $(\boldsymbol u, \theta)\in \mathcal C([0, t_1]; L^2_{\boldsymbol\tau}(\Omega))\times \mathcal C([0, t_1]; H_0^1(\Omega))$ of the 2D Euler-Boussinesq system \eqref{eq1.1.3}-\eqref{eq1.1.4} such that the following estimates hold:
	\begin{equation}\begin{cases}\label{eq2.0.4}
		\aiminnorm{\boldsymbol u}_{L^\infty(0, t_1; V)}  + \aiminnorm{\theta}_{L^\infty(0, t_1; H^1_0(\Omega))}+ \aiminnorm{\theta}_{L^2(0, t_1; H^2(\Omega))} \leq \mathcal Q_2,\\
		\aiminnorm{\boldsymbol u_t}_{L^2(0, t_1; L^{3/2}(\Omega))} + \aiminnorm{\theta_t}_{L^2(0, t_1; L^2(\Omega))}\leq \mathcal Q_2,
	\end{cases}\end{equation}
	where $\mathcal Q_2$ is a positive function defined by
	\[
	\mathcal Q_2:=\mathcal Q_2(t_1, \aiminnorm{\boldsymbol u_0}_{H^1}, \aiminnorm{\theta_0}_{H^1}, \aiminnorm{S}_{H^2}),
	\]
	which is increasing in all its arguments.
	
	Furthermore, if we additionally assume $\omega_0={\mathrm{curl}\,}\boldsymbol u_0\in L^{\infty}(\Omega)$, $\theta_0\in H^2(\Omega)$, and $S\in H^3(\Omega)$, then there exists a unique solution $(\boldsymbol u, \theta)$ of the 2D Euler-Boussinesq system \eqref{eq1.1.3}-\eqref{eq1.1.4} satisfying
	\begin{equation*}\begin{split}
	\omega={\mathrm{curl}\,}\boldsymbol u\in L^\infty(0, t_1; L^{\infty}(\Omega)),\quad  \theta\in \mathcal C([0, t_1]; H^2(\Omega))\cap L^2(0, t_1; H^3(\Omega)),\\
	 \theta_t\in L^\infty(0,t_1; L^2(\Omega))\cap L^2(0, t_1; H^1(\Omega)),
	\end{split}\end{equation*}
	and the estimates
	\begin{equation}\begin{split}\label{eq2.0.6}
		 \aiminnorm{\theta}_{L^\infty(0, t_1; H^2(\Omega)) } + \aiminnorm{\theta}_{L^2(0, t_1; H^3(\Omega)) }
		 \leq \mathcal Q_3,\\
	\aiminnorm{\theta_t}_{L^\infty(0, t_1; L^2(\Omega))} + \aiminnorm{\theta_t}_{L^2(0, t_1; H^1(\Omega)) }
	\leq \mathcal Q_3,\\
	\aiminnorm{\omega}_{ L^\infty(\Omega\times(0, t_1))} \leq \mathcal Q_4.
	\end{split}\end{equation}
	where $\mathcal Q_3$ and $\mathcal Q_4$ are  positive functions defined by
	\[
	\mathcal Q_3:=\mathcal Q_3(t_1, \aiminnorm{\omega_0}_{L^{4}}, \aiminnorm{\theta_0}_{H^2}, \aiminnorm{S}_{H^3}),\qquad
	\mathcal Q_4:=\mathcal Q_4(t_1, \aiminnorm{\omega_0}_{L^{\infty}(\Omega)}, \aiminnorm{\theta_0}_{H^2}, \aiminnorm{S}_{H^3}),
	\]
	which are increasing in all their arguments.	
\end{theorem}
\begin{remark}
We first note that the regularity of $\theta$ in Theorem~\ref{thm2.1} only depends on the $L^4$-norm of the initial vorticity $\omega_0$ and hence as in \cite{Yud95}, the estimate \eqref{eq1.3.2} can be used to show the uniqueness part of Theorem~\ref{thm2.1} under weaker assumption than $\omega_0\in L^\infty(\Omega)$, including unbounded initial vorticity with controlled growth rate of the $L^p$-norm of $\omega_0$ as $p\rightarrow \infty$. For instance, one can take
\[
\omega_0\in \bigcap_{1<p<\infty}L^p(\Omega),\qquad\text{ with } \qquad \sup_{p>e^e} \frac{ \aiminnorm{\omega_0}_{L^p} }{ \log\log p}  < \infty
\]
and see \cite[Section~5]{Yud95} for a precise definition of the class of allowed data.

\end{remark}

\section{Notations and preliminaries}\label{sec2}
Here and throughout this article, we will not distinguish the notations for vector and scalar function spaces whenever they are self-evident from the context. For $s\in\mathbb R$ and $1\leq p\leq\infty$, we denote by $W^{s,p}(\Omega)$ (resp. $H^s(\Omega)$) the classical Sobolev space of order $s$ on $\Omega$ with norm $\aiminnorm{\cdot}_{W^{s,p}}$ (resp. $\aiminnorm{\cdot}_{H^s}$), by $W^{s,p}_0(\Omega)$ (resp. $H_0^s(\Omega)$) the closure of $\mathcal D(\Omega)$ in the space $W^{s,p}(\Omega)$ (resp. $H^s(\Omega)$) when $s>0$,  and by $L^p(\Omega)$  the classical $L^p$-Lebesgue space with norm $\|\cdot\|_{L^p}$. For simplicity, we reserve the notation $\aiminnorm{\cdot}$for the $L^2$-norm.

In this article, we denote by $\mathcal Q_i(\cdot)$ ($i=1,2,\cdots$) the positive increasing functions in all their arguments, which may vary from line to line. The symbol $C$ denotes a generic positive constant, which may depend on the domain $\Omega$, but independent of the data $\boldsymbol u_0$, $\theta_0$, $S$ and the time $t_1$.

\subsection{$L^p$-tangential vector fields and Helmholtz decomposition}
We also introduce the $L^p$-tangential vector fields space as in \cite[Section~2.2.1]{BDT13}:
\[
L^p_{\boldsymbol\tau}(\Omega)=\aiminset{\boldsymbol u\in L^p(\Omega)\,:\,{\mathrm{div}\,}\boldsymbol u=0,\quad \boldsymbol u\cdot \boldsymbol n=0\text{ on }\partial\Omega},\qquad 1<p<\infty,
\]
and the smooth function space
\[
\mathcal V=\aiminset{\boldsymbol u\in\mathcal D(\Omega)\,:\, {\mathrm{div}\,}\boldsymbol u=0}.
\]
Let us denote by $\mathrm{P}_\Omega\,:\,\mathcal D(\Omega)\rightarrow \mathcal V$ the corresponding projection operator.
It is well known that $\mathrm{P}_\Omega$ extends to be a linear orthogonal operator on $L^2(\Omega)$ for general Lipschitz domains (see for example \cite[Theorem~I.1.4]{Tem01}). Recently, $\mathrm{P}_\Omega$ extends to be a bounded linear operator on $L^p(\Omega)$ for bounded convex domains (see \cite[Theorem~1.3]{GS10}) with $1<p<\infty$ and for general Lipschitz domains with the range $p\in(3/2-\epsilon, 3+\epsilon)$ (see \cite{FMM98}). Here, we would like to extend these results to the polygonal-like domains.
\begin{proposition}\label{prop1.1.1}
Let $1<p<\infty$ and $\Omega \subset \mathbb R^2$ be a polygonal-like domain. Then there holds
\begin{enumerate}[$\;(i)$]
	\item the space $\mathcal V$ is dense in $L^p_{\boldsymbol\tau}(\Omega)$;
	\item the operator $\mathrm{P}_\Omega$ is extended to be a bounded linear projection operator from $L^p(\Omega)$ to  $L_{\boldsymbol\tau}^p(\Omega)$ with the operator norm  only depending on $p$ and the domain $\Omega$;
	\item for each $\boldsymbol v \in L^p(\Omega)$, there exists $\pi$ belonging to the space $W^{1,p}(\Omega)$, unique up to an additive constant such that
		\begin{equation}\label{eq1.2.2}
		\mathrm{P}_\Omega^\perp \boldsymbol v:=(\mathrm{1}-\mathrm{P}_\Omega) \boldsymbol v= \nabla \pi.
		\end{equation}
		and with the estimate
		\[
		\max\big\{ \aiminnorm{ \mathrm{P}_\Omega \boldsymbol v}_{L^p},\; \aiminnorm{\nabla\pi}_{L^p} \big\} \leq C_{p, \Omega}\aiminnorm{\boldsymbol v}_{L^p},
		\]
		where $C_{p, \Omega}>0$ depends only on $p$ and the domain $\Omega$.
\end{enumerate}

\end{proposition}
We remark that we actually only need item $(ii)$ in the case when $p=3/2$ and $p=4$.
\begin{proof}[Proof of Proposition \ref{prop1.1.1}]
	Item $(i)$ is proved in \cite[Lemma~6.1]{GS10} and item $(ii)$ is a direct consequence of items $(i)$ and $(iii)$. We only need to show item $(iii)$ and the arguments are as follows. 	
	We first deduce from \cite[Theorem~6.3]{GS10} that one only needs to show that the following Neumann problem is uniquely solvable
	\begin{equation}\label{eq1.2.3}
		\Delta\Psi = 0,\text{ in }\Omega,\qquad \frac{\partial \Psi}{\partial\boldsymbol n}=g\in B_{-1/p}^{p}(\partial\Omega),\text{ on }\partial\Omega,\qquad \Psi\in W^{1, p}(\Omega),
	\end{equation}
	where $B_{-1/p}^{p}(\partial\Omega)$ is the dual of the Besov space $B_{1/p}^q(\Omega)$ with $q=p/(p-1)$.
	The Neumann problem \eqref{eq1.2.3} is solvable for $p\geq 2$ if the results in \cite[Theorem 1.2]{KS08} (see also \cite[Theorem 1.1]{GS10}) for bounded convex domains is also true for the polygonal-like domains  (see \cite[pp. 2161]{GS10}).
	Tracking the proof of \cite[Theorem 1.2]{KS08}, we see that the convexity of the domain is only used in \cite[Lemmas~4.1-4.2]{KS08}, which may fail for general Lipschitz domains. But for the polygonal-like domains, we infer from \cite[Theorem~3.1.1.2]{Gri85} that \cite[Lemma~4.1]{KS08} still holds and so does \cite[Lemma~4.2]{KS08}, which together show that \cite[Theorem 1.2]{KS08} is still valid for the polygonal-like domains. Therefore, we can conclude that the Helmholtz decomposition \eqref{eq1.2.2} in $L^p(\Omega)$ is valid for $p\geq 2$ and hence for all $1<p<\infty$ by duality (see \cite[Lemma~6.2]{GS10}).
\end{proof}

\subsection{The Dirichlet problem and the Biot-Savart law}
Let $F=\mathrm{G}_\Omega f$ be the solution of the Dirichlet problem
\begin{equation}\label{eq1.3.1}
-\Delta F = f,\qquad F|_{\partial\Omega}=0.	
\end{equation}
The Lax-Milgram lemma tells us that if $f\in H^{-1}(\Omega)$, then there exists a unique $F\in H_0^1(\Omega)$ denoted by $\mathrm{G}_\Omega f$ satisfying \eqref{eq1.3.1} in the distributional sense.
If we further assume $f\in L^p(\Omega)$ with $p\geq 2$, then the  elliptic regularity result in \cite{Gri85} and improved in \cite[Theorem~1]{DT13} for the polygonal-like domains with maximum aperture $\max \alpha_j\leq \pi/2$ guarantees that $\mathrm{G}_\Omega f$ still has two derivatives in $L^p(\Omega)$ and the following estimate holds:
\begin{equation}\label{eq1.3.2}
	\mathrm{G}_\Omega f\in W^{2,p}(\Omega)\cap W_0^{1,p}(\Omega),\qquad
	\aiminnorm{\mathrm{G}_\Omega f}_{ W^{2,p} }\leq C_{\Omega} p \aiminnorm{f}_{L^p},\quad 2\leq p<\infty,
\end{equation}
where $C_{\Omega}$ only depends on $\Omega$. We also infer from \cite[Theorem~5.1.1.4]{Gri85} that
\begin{equation}\label{eq1.3.3}
	\aiminnorm{\mathrm{G}_\Omega f}_{H^3} \leq C_{\Omega}\aiminnorm{f}_{H^1},\qquad \forall f\in H_0^1(\Omega).
\end{equation}

We now set
\[
\boldsymbol{\mathrm{K}}_\Omega:=\nabla^{\perp}(\mathrm{G}_\Omega f),\qquad f\in H^{-1}(\Omega),
\]
where $\nabla^{\perp}=(\partial_y, -\partial_x)$. Then since  $\Omega$ is a polygonal-like domain with maximum aperture $\max \alpha_j\leq \pi/2$, there holds for all $2\leq p<\infty$,
\begin{equation}\label{eq1.3.4}
\begin{cases}
	\boldsymbol{\mathrm{K}}_\Omega \in \mathscr L( H^{-1}(\Omega),\; L^2_{\boldsymbol\tau}(\Omega)),  \\
	\boldsymbol{\mathrm{K}}_\Omega \in \mathscr L( L^p(\Omega),\; W^{1,p}(\Omega)\cap L^2_{\boldsymbol\tau}(\Omega)),\\
	\boldsymbol{\mathrm{K}}_\Omega \in \mathscr L( H_0^1(\Omega),\; H^2(\Omega)\cap L^2_{\boldsymbol\tau}(\Omega)).
\end{cases}
\end{equation}
To prove \eqref{eq1.3.4}, due to the regularity estimates \eqref{eq1.3.2}-\eqref{eq1.3.3}, we only need to verify that
\[
{\mathrm{div}\,} \boldsymbol{\mathrm{K}}_\Omega f = 0,\qquad \boldsymbol{\mathrm{K}}_\Omega f\cdot\boldsymbol n=0,\qquad \forall f\in H^{-1}(\Omega),
\]
which follows from the fact that $\mathcal V$ is dense in $L^2_{\boldsymbol \tau}(\Omega)$ and the following identity
\begin{equation*}\begin{split}
\aimininner{\boldsymbol{\mathrm{K}}_\Omega f}{\nabla \varphi}_\Omega &= \langle\nabla ( \mathrm{G}_\Omega f), \nabla^\perp \varphi\rangle_\Omega \\
&= -\aimininner{ \mathrm{G}_\Omega f}{ {\mathrm{div}\,} \nabla^\perp \varphi}_\Omega + \int_{\partial\Omega} (\mathrm{G}_\Omega f )\nabla^\perp \varphi \cdot \mathbf{n}=0, \quad \forall\varphi \in \mathcal D(\overline{\Omega}).	
\end{split}\end{equation*}

\subsection{Elliptic regularity at $p\rightarrow\infty$}
In order to extend the elliptic regularity \eqref{eq1.3.2} to the end point when $p\rightarrow\infty$, one needs to work with the Orlicz spaces, where the elliptic regularity result in these spaces is recently proved in \cite{DT13} for the polygonal-like domains.

\subsubsection{The Orlicz spaces.}
Here, we briefly recall some preliminaries on the Orlicz spaces and see \cite{RR91, Wil08} for more details.
A function $\gamma\,:\,[0, \infty]\mapsto [0, \infty]$ is said to be a Young function if
\begin{enumerate}[$\;(1)$]
	\item $\gamma$ is increasing and $\gamma(0)=0$, $\lim_{s\rightarrow \infty}\gamma(s)=\infty$;
	\item $\gamma$ is convex lower-semicontinuous $[0, \infty]$-valued function on $\mathbb R$;
	\item $\gamma$ is non-trivial, that is there exists a number $0<s_0<\infty$ such that $0<\gamma(s_0)<\infty$.
\end{enumerate}
The convex conjugate $\gamma^*$ of a $\gamma$ is defined by:
\[
\gamma^*(t):=\sup\{ st - \gamma(s),\qquad s\geq 0\},
\]
and one can show that $\gamma$ is a Young function if and only if $\gamma^*$ is a Young function.
The convex conjugacy allows us to obtain the Orlicz space version of H\"older's inequality.
The typical examples of Young functions are
\[
	\gamma_p(s) = s^p,\quad p> 1,\qquad\qquad
	\gamma_{\mathrm{exp}}(s) = e^s - 1.
\]
and their corresponding convex conjugates:
\[
\gamma_p(t) = t^{p'},\quad p'=\frac{p}{p-1},\qquad\qquad
	\gamma_{\mathrm{exp}}^*(t) = \begin{cases}
		t\ln t - t + 1, &\forall t\geq 1,\\
		0, &0\leq t\leq 1.
	\end{cases}.
\]
We now define the Orlicz spaces $L^\gamma(\Omega)$ to the the set of all measurable functions such that the Luxemburg norm is finite, that is
\[
\aiminnorm{f}_{L^{\gamma}}:=\inf\{ \lambda>0\,:\, \int_\Omega \gamma(\aiminabs{f}/\lambda ){\mathrm{d}x}{\mathrm{d}y} \leq 1 \} <\infty.
\]
One can easily verify $L^{\gamma_p}(\Omega)=L^p(\Omega)$ for all $p> 1$ and we also have the following result.
\begin{proposition}\label{prop1.1.2}
	Let $\gamma=\gamma(s)$ be a Young function. Then the space $L^\gamma(\Omega)$ with the norm $\aiminnorm{\cdot}_{\gamma}$ is a Banach space and for all $f\in L^{\gamma}(\Omega)$ and $f\in L^{\gamma^*}(\Omega)$, there holds
	\[
	\int_\Omega \aiminabs{f}\aiminabs{g}{\mathrm{d}x}{\mathrm{d}y}\leq 2\aiminnorm{f}_{L^{\gamma}}\aiminnorm{g}_{L^{\gamma^*}}.
	\]
\end{proposition}
In this article, we are interested in the Young function $\gamma_{\mathrm{exp}}$ and its convex conjugate $\gamma_{\mathrm{exp}}^*$. Direct calculation shows that
\[
\gamma_{\mathrm{exp}}^*(t) \leq \frac{ t^{1+\epsilon} }{ \epsilon },\qquad \forall\,t\geq 0,\quad \forall\,0<\epsilon\leq 1,
\]
which permits us to conclude the following
\begin{equation}\label{eq1.5.3}
	\aiminnorm{f}_{L^{\gamma_{\mathrm{exp}}^*}} \leq \epsilon^{-1/(1+\epsilon)}\aiminnorm{f}_{L^{1+\epsilon}},\qquad\, \forall\, 0<\epsilon\leq 1.
\end{equation}

\subsubsection{Elliptic regularity.}
The following result, which we borrow from  \cite[Theorem~1]{DT13}, gives an analogue to the elliptic regularity \eqref{eq1.3.2} at the end point $p\rightarrow\infty$.
In our case when $\Omega$ is a polygonal-like domain with maximum aperture $\max \alpha_j\leq \pi/2$, there holds
\begin{equation}\label{eq1.5.5}
	\aiminnorm{D^2\mathrm{G}_\Omega f}_{L^{\gamma_{\mathrm{exp}}}} \leq C_\Omega \aiminnorm{f}_{L^\infty}.
\end{equation}
where the constant $C_\Omega>0$ depends only on the domain $\Omega$.

\begin{remark}
In the case when $\Omega$ is a polygonal-like domain with the aperture $\alpha_j$ of the form $\frac{\pi}{k}$ for some integer $k\geq 2$, the elliptic regularity results in \cite[Proposition 3.1 and Remark 5.2]{BDT13} tell us that
\begin{equation}\label{eq1.5.4}
\aiminnorm{ D^2\mathrm{G}_\Omega f}_{\mathrm{bmo}_r(\Omega)}  \leq C_\Omega \aiminnorm{f}_{\mathrm{bmo}_z(\Omega) },
\end{equation}
which is a stronger inequality than \eqref{eq1.5.5}. For a definition of the local $\mathrm{bmo}_\star(\Omega)$ ($\star=z, r$) spaces, see \cite[Section~3.1]{BDT13} or \cite{CDS99}.
The extension of \eqref{eq1.5.4} to general polygonal-like domains is still an open problem, to the best of our knowledge (see also \cite[Remark 1.1]{DT13}). \qed
\end{remark}

\section{Approximate solutions}\label{sec3}
Inspired by \cite{Bar72} where the vanishing viscosity method is applied to the 2D Euler equations in a bounded smooth domain, we here utilize the same method to study the 2D Euler-Boussinesq system. Hence, we introduce the 2D Boussinesq system with full viscosity $0<\nu\leq 1$ and diffusivity $\kappa=1$:
\begin{equation}\begin{cases}\label{eq2.1.3}
	\partial_t\boldsymbol u_\nu  -\nu\Delta\boldsymbol u_\nu+ \boldsymbol u_\nu\cdot \nabla \boldsymbol u_\nu + \nabla\pi_\nu = \theta_\nu \boldsymbol e_2 + S\boldsymbol e_2,\quad \boldsymbol e_2=(0,1),\\
	{\mathrm{div}\,}\boldsymbol u_\nu=0,\\
	\partial_t \theta_\nu - \Delta\theta_\nu + \boldsymbol u_\nu\cdot\nabla\theta_\nu +\boldsymbol u_\nu\cdot\nabla S =  \Delta S, \\
\end{cases}\end{equation}
with the initial and boundary conditions
\begin{equation}\begin{split}\label{eq2.1.4}
	\boldsymbol u_\nu(0)=\boldsymbol u_0,\quad\quad &\theta_\nu(0)=\theta_0,\qquad\text{ in }\Omega,\\
	\boldsymbol u_\nu\cdot \boldsymbol n=0,\quad\frac{\partial(\boldsymbol u_\nu\cdot\boldsymbol\tau)}{\partial\boldsymbol n}=0,\quad\quad &\theta_\nu=0,\qquad \text{ on }\partial\Omega.\\
\end{split}\end{equation}
The existence and uniqueness of a global strong solution $(\boldsymbol u_\nu, \theta_\nu)$ of the 2D Boussinesq system \eqref{eq2.1.3}-\eqref{eq2.1.4} in the polygonal-like domain $\Omega$ are classical obtained using the Galerkin procedure, see for example \cite{FMT87, Tem97}. Here, we only need to prove some uniform estimates independent of $\nu$.
\begin{lemma}\label{lem2.1.1}
	Assume that $S\in H^2(\Omega)$, $\boldsymbol u_0\in V$, and $\theta_0\in H_0^1(\Omega)$. Then the solutions  $(\boldsymbol u_\nu, \theta_\nu)$ of \eqref{eq2.1.3}-\eqref{eq2.1.4} satisfy the following estimates
	\begin{equation}\begin{split}\label{eq2.1.6}
		\sup_{t\in[0, t_1]}(\aiminnorm{\boldsymbol u_\nu(t)}_{H^1}^2 + \aiminnorm{\theta_\nu(t)}_{H^1}^2) + \int_0^{t_1}\aiminnorm{\Delta\theta_\nu(t)}^2{\mathrm{d}t} + \nu\int_0^{t_1}\aiminnorm{\boldsymbol u_\nu(t)}_{H^2}^2{\mathrm{d}t} &\leq \mathcal Q_2,\\
		\aiminnorm{\partial_t\boldsymbol u_\nu}_{L^2(0, t_1; L^{3/2}(\Omega))} + \aiminnorm{\partial_t\theta_\nu}_{L^2(0, t_1; L^{2}(\Omega))} &\leq \mathcal Q_2,
	\end{split}\end{equation}
	where $\mathcal Q_2$ is a positive function independent of $\nu$ $(0<\nu\leq 1$) defined by
	\[
	\mathcal Q_2:=\mathcal Q_2(t_1, \aiminnorm{\boldsymbol u_0}_{H^1}, \aiminnorm{\theta_0}_{H^1}, \aiminnorm{S}_{H^2}),
	\]
	which is increasing in all its arguments.
\end{lemma}
In the sequel, the symbol $C$ denotes a generic positive constant, which may depend on the domain $\Omega$ and vary from line to line.
\begin{proof}[Proof of Lemma \ref{lem2.1.1}]
	For the sake of simplicity, we write $(\boldsymbol u, \theta)$ instead of $(\boldsymbol u_\nu, \theta_\nu)$ by  dropping the subscript $\nu$ in the following proof.
	Multiplying \eqref{eq2.1.3}$_1$ with $\boldsymbol u$, integrating in $L^2(\Omega)$, and using the Cauchy-Schwartz inequality, we obtain	
	\begin{equation}\label{eq2.1.7}
		\frac12\frac{\mathrm{d}}{\mathrm{d}t}\aiminnorm{\boldsymbol u}^2 + \nu\aiminnorm{\nabla\boldsymbol u}^2 \leq  \aiminnorm{\boldsymbol u}^2 + \frac12\aiminnorm{\theta}^2 +\frac12\aiminnorm{S}^2.
	\end{equation}

	Taking the  inner product \eqref{eq2.1.3}$_3$ with $\theta$ in $L^2(\Omega)$ and using H\"older's inequality and the Sobolev embedding, we find
	\begin{equation*}\begin{split}
	\frac12\frac{\mathrm{d}}{\mathrm{d}t}\aiminnorm{\theta}^2 + \aiminnorm{\nabla\theta}^2
	&\leq \aiminnorm{\boldsymbol u}\aiminnorm{\nabla S}_{L^4}\aiminnorm{\theta}_{L^4} +\aiminnorm{\Delta S}\aiminnorm{\theta}\\
	&\leq C\aiminnorm{\boldsymbol u}\aiminnorm{S}_{H^2}\aiminnorm{\nabla\theta} + \aiminnorm{S}_{H^2}\aiminnorm{\theta},
	\end{split}\end{equation*}
	which, by  Young's inequality, yields
	\begin{equation}\label{eq2.1.8}
		\frac12\frac{\mathrm{d}}{\mathrm{d}t}\aiminnorm{\theta}^2 + \aiminnorm{\nabla\theta}^2
		\leq C\aiminnorm{S}_{H^2}^2 \aiminnorm{\boldsymbol u}^2 + \frac12\aiminnorm{S}_{H^2}^2 + \frac12\aiminnorm{\theta}^2 + \frac12\aiminnorm{\nabla\theta}^2,
	\end{equation}
	where the constant $C>0$ only depends on the domain $\Omega$.

	Summing \eqref{eq2.1.7} and \eqref{eq2.1.8} together, we arrive at
	\[
	\frac{\mathrm{d}}{\mathrm{d}t}(\aiminnorm{\boldsymbol u}^2 + \aiminnorm{\theta}^2) + 2\nu\aiminnorm{\nabla\boldsymbol u}^2 + \aiminnorm{\nabla\theta}^2
	\leq C\aiminnorm{S}_{H^2}^2 + C(\aiminnorm{S}_{H^2}^2 + 1)( \aiminnorm{\boldsymbol u}^2 + \aiminnorm{\theta}^2 ).
	\]
	Applying the Gronwall lemma, we obtain
	\begin{equation}\begin{split}\label{eq2.1.9}
		\sup_{t\in[0, t_1]}(\aiminnorm{\boldsymbol u(t)}^2+\aiminnorm{\theta(t)}^2)
		+2\nu\int_0^{t_1}\aiminnorm{\nabla\boldsymbol u(t)}^2{\mathrm{d}t} + \int_0^{t_1}\aiminnorm{\nabla\theta(t)}^2{\mathrm{d}t}\\
		\leq e^{C t_1 (\aiminnorm{S}_{H^2}^2+1)} \big(\aiminnorm{\boldsymbol u_0}^2+\aiminnorm{\theta_0}^2 +C t_1\aiminnorm{S}_{H^2}^2\big).
	\end{split}\end{equation}

	In order to find the uniform $H^1$-estimate, we need to use vorticity formulation together with the Biot-Savart law. Let $\omega={\mathrm{curl}\,}\boldsymbol u=\partial_x u_2 - \partial_y u_1$, then the vorticity $\omega$ satisfies
	\begin{equation}\label{eq2.1.11}
		\partial_t\omega -\nu\Delta\omega + \boldsymbol u\cdot\nabla\omega = \partial_x\theta+\partial_x S,
	\end{equation}
	with the Dirichlet boundary condition
	\[
	\omega = 0,\quad\text{ on }\partial\Omega.
	\]
	That $\omega$ satisfies the homogeneous Dirichlet boundary condition is from the boundary conditions \eqref{eq2.1.4}$_2$ and the calculation:
	\[
	\omega={\mathrm{curl}\,}\boldsymbol u={\mathrm{curl}\,}( (\boldsymbol u\cdot\boldsymbol\tau)\boldsymbol\tau + (\boldsymbol u\cdot\boldsymbol n)\boldsymbol n)=\frac{\partial(\boldsymbol u\cdot\boldsymbol n)}{\partial \boldsymbol\tau}  - \frac{\partial(\boldsymbol u\cdot\boldsymbol \tau)}{\partial\boldsymbol n}=0,\quad\text{ on }\partial\Omega.
	\]
	By the Biot-Savart law \eqref{eq1.3.4} (see also \cite{Kat67, Bar72}), we have
	\begin{equation}\label{eq2.1.12}
		\aiminnorm{\boldsymbol u}_{H^1}^2 \leq C\aiminnorm{\omega}_{L^2}^2,\qquad
		\aiminnorm{\boldsymbol u}_{H^2}^2 \leq C\aiminnorm{\omega}_{H^1}^2\leq C\aiminnorm{\nabla\omega}^2,
	\end{equation}
	where the Poincar\'e inequality is employed for the last inequality.
	
	Taking the inner product \eqref{eq2.1.11} with $\omega$ in $L^2(\Omega)$ and using the Cauchy-Schwartz inequality gives
	\begin{equation}\label{eq2.1.13}
		\frac12\frac{\mathrm{d}}{\mathrm{d}t}\aiminnorm{\omega}^2 + \nu\aiminnorm{\nabla\omega}^2\leq \aiminnorm{\omega}^2  + \frac12\aiminnorm{\partial_x\theta}^2 + \frac12\aiminnorm{\partial_xS}^2.
	\end{equation}
	Taking the  inner product \eqref{eq2.1.3}$_3$ with $-\Delta\theta$ in $L^2(\Omega)$ and using H\"older's and Ladyzhenskaya's inequalities, we arrive at
	\begin{equation*}\begin{split}
		\frac12\frac{\mathrm{d}}{\mathrm{d}t}\aiminnorm{\nabla\theta}^2 + \aiminnorm{\Delta\theta}^2
		\leq \aiminnorm{\boldsymbol u}_{L^4}\aiminnorm{\nabla\theta}_{L^4}\aiminnorm{\Delta\theta} + \aiminnorm{\boldsymbol u}_{L^4}\aiminnorm{\nabla S}_{L^4}\aiminnorm{\Delta\theta} + \aiminnorm{\Delta S}\aiminnorm{\Delta\theta}\\
		\leq C\aiminnorm{\boldsymbol u}^{1/2}\aiminnorm{\boldsymbol u}_{H^1}^{1/2}\aiminnorm{\nabla\theta}^{1/2}\aiminnorm{\Delta\theta}^{3/2}
		+C\aiminnorm{\boldsymbol u}_{H^1}\aiminnorm{S}_{H^2}\aiminnorm{\Delta\theta} + \aiminnorm{ S}_{H^2}\aiminnorm{\Delta\theta},
	\end{split}\end{equation*}
	which, by Young's inequality, yields
	\begin{equation}\begin{split}\label{eq2.1.14}
	\frac12\frac{\mathrm{d}}{\mathrm{d}t}\aiminnorm{\nabla\theta}^2 + \aiminnorm{\Delta\theta}^2	\leq
	C\aiminnorm{\boldsymbol u}^{2} \aiminnorm{\nabla\theta}^2\aiminnorm{\boldsymbol u}_{H^1}^2 + C\aiminnorm{S}_{H^2}^2\aiminnorm{\boldsymbol u}_{H^1}^2+C\aiminnorm{S}^2_{H^2}+ \frac12\aiminnorm{\Delta\theta}^2.
	\end{split}\end{equation}
	
	Combining the estimates \eqref{eq2.1.13} and \eqref{eq2.1.14} and using \eqref{eq2.1.12}, we see that
	\begin{equation*}\begin{split}
		\frac{\mathrm{d}}{\mathrm{d}t}(\aiminnorm{\omega}^2 + \aiminnorm{\nabla\theta}^2) +& 2\nu\aiminnorm{\nabla\omega}^2 + \aiminnorm{\Delta\theta}^2\\
		&\leq C(\aiminnorm{\boldsymbol u}^2\aiminnorm{\nabla\theta}^2 + \aiminnorm{S}^2_{H^2} + 1 )(\aiminnorm{\omega}^2 + \aiminnorm{\nabla\theta}^2) + C\aiminnorm{S}_{H^2}^2.
	\end{split}\end{equation*}
	Applying the Gronwall lemma, we obtain
	\begin{equation}\begin{split}\label{eq2.1.16}
		&\sup_{t\in[0, t_1]}(\aiminnorm{\omega(t)}^2 + \aiminnorm{\nabla\theta(t)}^2) + 2\nu\int_0^{t_1}\aiminnorm{\nabla\omega(t)}^2{\mathrm{d}t} + \int_0^{t_1}\aiminnorm{\Delta\theta(t)}^2{\mathrm{d}t}\\
		&\leq (\aiminnorm{\omega_0}^2 + \aiminnorm{\nabla\theta_0}^2 + Ct_1\aiminnorm{S}_{H^2}^2)\exp\big\{ C( \sup_{t\in[0, t_1]}\aiminnorm{\boldsymbol u(t)}^2\int_0^{t_1}\aiminnorm{\nabla\theta}^2{\mathrm{d}t} + t_1\aiminnorm{S}^2_{H^2} + t_1) \big\},
	\end{split}\end{equation}
	 which implies the first inequality in \eqref{eq2.1.6} by taking the estimate \eqref{eq2.1.9} and the Biot-Savart law \eqref{eq2.1.12} into consideration.
	
	 We now turn to the second inequality in \eqref{eq2.1.6} on the time-derivatives of $(\boldsymbol u, \theta)$.
	 Applying the projection operator $\mathrm{P}_\Omega$ to \eqref{eq2.1.3}$_1$ gives the identity
	 \begin{equation*}
		 \partial_t\boldsymbol u=\mathrm{P}_\Omega(\nu\Delta\boldsymbol u - \boldsymbol u\nabla\boldsymbol u + \boldsymbol f),
	 \end{equation*}
	 where $\boldsymbol f=\theta\boldsymbol e_2 + S\boldsymbol e_2$. Noticing that by \eqref{eq2.1.16} and \eqref{eq2.1.12}, $\nu\aiminnorm{\Delta\boldsymbol u}_{L^2(0,t_1; L^2(\Omega))}$ is uniformly bounded independent of $\nu$ ($0<\nu\leq 1$) and using the estimate \eqref{eq2.1.16} again, the arguments for \eqref{eqa.13} in the case when $p=2$ tell that
	 \begin{equation}\label{eq2.1.20}
	 \partial_t\boldsymbol u\in L^2(0, t_1; L^s(\Omega)), \qquad \forall 1\leq s<2.
	 \end{equation}
	 Hence,
	 \[
	 \partial_t\boldsymbol u \in L^2(0, t_1; L^{3/2}(\Omega)).
	 \]
	 Regarding $\partial_t\theta$, we take a test function $\tilde\theta\in L^2(0, t_1; L^2(\Omega))$ with norm at most $1$ and find from \eqref{eq2.1.3}$_3$ that
	 \[
	 |\aimininner{\partial_t\theta}{\tilde\theta}|\leq \aiminnorm{\Delta\theta}\aiminnorm{\tilde\theta} + \aiminnorm{\boldsymbol u}_{L^4}\aiminnorm{\nabla\theta}_{L^4}\aiminnorm{\tilde\theta} +  \aiminnorm{\boldsymbol u}_{L^4}\aiminnorm{\nabla S}_{L^4}\aiminnorm{\tilde\theta} + \aiminnorm{\Delta S}\aiminnorm{\tilde\theta}.
	 \]
	 Thanks to  the uniform estimate \eqref{eq2.1.6}$_1$ again, we obtain
	 \[
	 \partial_t \theta \in L^2(0, t_1; L^{2}(\Omega) ).
	 \]
	 Therefore, we finished proving the inequality \eqref{eq2.1.6}. This ends the proof of Lemma \ref{lem2.1.1}.
\end{proof}

\section{Proof of Theorem \ref{thm2.1}}\label{sec4}
The goal here is to prove the main result of this article and we divide it to three parts. We first prove the existence of \emph{suitable weak solution} for the 2D Euler-Boussinesq system \eqref{eq1.1.3}-\eqref{eq1.1.4}, then improve the regularity of the solution, and finally show the uniqueness of the solution.

\subsection{Existence of \emph{suitable weak solution}}
Thanks to the fact that the estimate \eqref{eq2.1.6} in Lemma \ref{lem2.1.1} is independently of $\nu$, we infer the existence of a couple $(\boldsymbol u, \theta)$ such that
\begin{equation*}\begin{split}
\boldsymbol u\in L^\infty(0, t_1; V),\qquad \partial_t\boldsymbol u\in L^2(0, t_1; L^{3/2}(\Omega)),\\
\theta\in L^\infty(0, t_1; H^1_0(\Omega))\cap L^2(0, t_1; H^2(\Omega)),\qquad \partial_t\theta\in L^2(0, t_1; L^{2}(\Omega)),
\end{split}\end{equation*}
for which the following convergences up to not relabeled subsequences  are true.
\begin{itemize}
	\item $\boldsymbol u_\nu \rightarrow \boldsymbol u$ weakly-$*$ in $L^\infty(0, t_1; V)$ and $\partial_t\boldsymbol u_\nu\rightarrow \partial_t\boldsymbol u$ weakly in $L^2(0, t_1; L^{3/2}(\Omega))$. As a consequence (see e.g. \cite{Lio69}), $\boldsymbol u_\nu\to \boldsymbol u$ strongly in $L^2(0,t_1; L^6(\Omega))$.
	\item $\theta_\nu \rightarrow \theta$ weakly-$*$ in $L^\infty(0, t_1; H^1_0(\Omega))$ and weakly in $L^2(0, t_1; H^2(\Omega))$, and $\partial_t\theta_\nu\rightarrow \partial_t\theta$ weakly in $L^2(0, t_1; L^{2}(\Omega))$. Therefore, $\theta_\nu\rightarrow \theta$ strongly in $L^2(0, t_1; H_0^1(\Omega))$.
\end{itemize}
By interpolation (see e.g. \cite{LM72}), we also have $\boldsymbol u\in \mathcal C([0, t_1]; L^2_{\boldsymbol\tau}(\Omega))$ and $\theta\in \mathcal C([0, t_1]; H_0^1(\Omega))$. The estimate  \eqref{eq2.0.4} in Theorem \ref{thm2.1} directly follows  from the uniform estimate \eqref{eq2.1.6} which is independent of $\nu$.

Let $\tilde{\boldsymbol u}\in L_{\boldsymbol\tau}^3(\Omega)$, $\tilde\theta\in L^2(\Omega)$ and $\psi, \varphi\in\mathcal C^1([0, t_1])$ with $\psi(t_1)=\varphi(t_1)=0$, we then take the $L^2$-inner product \eqref{eq2.1.3} with $(\tilde{\boldsymbol u}\psi(t), \tilde\theta\varphi(t))$, integrate in time from $0$ to $t_1$ and integrate by parts for the first term; we arrive at
\begin{equation}\begin{split}\label{eq2.2.3}
	-\int_0^{t_1}\aimininner{\boldsymbol u_\nu(t)}{\tilde{\boldsymbol u}}\psi'(t){\mathrm{d}t} - \nu\int_0^{t_1}\aimininner{\Delta\boldsymbol u_\nu(t)}{\tilde{\boldsymbol u}}\psi(t){\mathrm{d}t}
	+\int_0^{t_1}\aimininner{\boldsymbol u_\nu(t)\cdot\nabla\boldsymbol u_\nu(t)}{\tilde{\boldsymbol u}}\psi(t){\mathrm{d}t}\\
	=\aimininner{\boldsymbol u_0}{\tilde{\boldsymbol u}}\psi(0) + \int_0^{t_1}\aimininner{\boldsymbol \theta_\nu(t)\boldsymbol e_2 + S\boldsymbol e_2}{\tilde{\boldsymbol u}}\psi(t){\mathrm{d}t},
\end{split}\end{equation}
\begin{equation}\begin{split}\label{eq2.2.4}
	-\int_0^{t_1}\aimininner{\theta_\nu(t)}{\tilde\theta}\varphi(t){\mathrm{d}t} -\int_0^{t_1}\aimininner{\Delta\theta_\nu(t)}{\tilde\theta}\varphi(t){\mathrm{d}t}+\int_0^{t_1}\aimininner{\boldsymbol u_\nu(t)\cdot\nabla\theta_\nu(t)}{\tilde\theta}\varphi(t){\mathrm{d}t}\\
	+\int_0^{t_1}\aimininner{\boldsymbol u_\nu(t)\cdot\nabla S}{\tilde\theta}\varphi(t){\mathrm{d}t}
	=\aimininner{\theta_0}{\tilde\theta}\varphi(0)  + \int_0^{t_1}\aimininner{\Delta S}{\tilde\theta}\varphi(t){\mathrm{d}t}.
\end{split}\end{equation}
Thanks to the uniform estimate \eqref{eq2.1.6} in Lemma \ref{lem2.1.1}, the second term in \eqref{eq2.2.3} converges to zero, that is
\[
\nu\int_0^{t_1}\aimininner{\Delta\boldsymbol u_\nu(t)}{\tilde{\boldsymbol u}}\psi(t){\mathrm{d}t} \rightarrow 0,\qquad\text{ as }\nu\rightarrow 0.
\]
The other linear terms in \eqref{eq2.2.3}-\eqref{eq2.2.4} converge to their corresponding limits in a straightforward manner due to the above convergences. The nonlinear term in \eqref{eq2.2.3} can be written as
\[
\int_0^{t_1}\aimininner{(\boldsymbol u_\nu(t) - \boldsymbol u(t))\cdot\nabla\boldsymbol u_\nu(t)}{\tilde{\boldsymbol u}}\psi(t){\mathrm{d}t}
+\int_0^{t_1}\aimininner{\boldsymbol u(t)\cdot\nabla\boldsymbol u_\nu(t)}{\tilde{\boldsymbol u}}\psi(t){\mathrm{d}t},
\]
and the first term above converges to zero due to the strong convergence of $\boldsymbol u_\nu\rightarrow\boldsymbol u$ in $L^2(0, t_1; L^6(\Omega))$ and the uniform boundedness of $\boldsymbol u_\nu$ in $L^\infty(0, t_1; V)$, and the second term above converges to
\[
\int_0^{t_1}\aimininner{ \boldsymbol u(t)\cdot\nabla\boldsymbol u(t)}{\tilde{\boldsymbol u}}\psi(t){\mathrm{d}t},
\]
because of the weak-$*$ convergence of $\boldsymbol u_\nu\rightarrow \boldsymbol u$ in $L^\infty(0, t_1; V)$.
The convergence of the nonlinear term in \eqref{eq2.2.4} is similar and simpler since we have better convergence for $\theta_\nu$. Therefore, we completed the proof of existence part of Theorem~\ref{thm2.1}.

\subsection{Regularity}
Now, if we assume additionally $\omega_0={\mathrm{curl}\,}\boldsymbol u_0\in L^{\infty}(\Omega)$ and $\theta_0\in H^2(\Omega)$, $S\in H^3(\Omega)$, then we are able to prove $L^\infty$-estimate for the vorticity $\omega$ and hence the $L^p$-estimate for the velocity $\boldsymbol u$ for any $1<p<\infty$ and the uniform $H^2$ and the time average of $H^3(\Omega)$-estimate for $\theta$.

For proving the $L^\infty$-estimate of the vorticity $\omega$, we require the $L^2(0, t_1; W^{1,\infty}(\Omega))$-estimate of the forcing term $\theta + S$ for the Euler equations. Hence, we first need a $L^2(0, t_1; W^{1,\infty}(\Omega))$-regularity of $\theta$, which turns to be the $L^2(0, t_1; H^3(\Omega))$-regularity for $\theta$.
To obtain the time average of $H^3$-regularity for $\theta$, we at least need the uniform $W^{1,4}(\Omega)$-estimate for the velocity $\boldsymbol u$. In conclusion, the plan for this subsection is as follows. We first derive the uniform $W^{1,4}$-estimate for $\boldsymbol u$, then show the uniform $H^2$ and the time average of $H^3$-estimate for $\theta$, and finally prove the $L^\infty$-estimate for the vorticity $\omega$.

From the Ladyzhenskaya's inequality
\[
\aiminnorm{f}_{L^4}^4 \leq \aiminnorm{f}_{L^2}^2\aiminnorm{f}_{H^1}^2,
\]
we deduce that $\theta\in L^4(0, t_1; W^{1,4}(\Omega))$ and by the Sobolev embedding, $S\in W^{1,4}(\Omega)$. Currently, if we only assume $\omega_0={\mathrm{curl}\,}\boldsymbol u_0\in L^{4}(\Omega)$, then applying Proposition \ref{propa.1} with $\boldsymbol f=\theta\boldsymbol e_2 + S\boldsymbol e_2\in  L^4(0, t_1; W^{1,4}(\Omega))$ and $p=4$ shows that
\begin{equation}\begin{split}\label{eq2.2.8}
&\hspace{40pt}\boldsymbol u\in L^\infty(0, t_1; W^{1,4}(\Omega)),\qquad \partial_t\boldsymbol u\in L^4(0, t_1; L^4(\Omega)),\\
&\aiminnorm{\boldsymbol u}_{L^\infty(0, t_1; W^{1,4}(\Omega)) } + \aiminnorm{\partial_t\boldsymbol u}_{L^4(0, t_1; L^4(\Omega))}
\leq \mathcal Q_3(t_1, \aiminnorm{\omega_0}_{L^4}, \aiminnorm{\theta_0}_{H^1}, \aiminnorm{S}_{H^2}).
\end{split}\end{equation}

To obtain the $H^2$-regularity of $\theta$, we differentiate \eqref{eq1.1.3} in time $t$ to find
\begin{equation*}
	\partial_t\theta_t - \Delta\theta_t + \boldsymbol u\cdot\nabla\theta_t + \boldsymbol u_t\cdot\nabla\theta + \boldsymbol u_t\cdot\nabla S = 0.
\end{equation*}
Applying the standard energy estimate, we arrive at
\begin{equation}\begin{split}\label{eq2.2.10}
	\frac12\frac{\mathrm{d}}{\mathrm{d}t} \aiminnorm{\theta_t}^2 + \aiminnorm{\nabla\theta_t}^2
	\leq \aiminnorm{\boldsymbol u_t}_{L^4}\aiminnorm{\nabla\theta}_{L^4}\aiminnorm{\theta_t} +  \aiminnorm{\boldsymbol u_t}_{L^4}\aiminnorm{\nabla S}_{L^4}\aiminnorm{\theta_t},
\end{split}\end{equation}
and by Ladyzhenskaya's inequality and  the Sobolev embedding, the right-hand side is bounded by
\[
C\aiminnorm{\boldsymbol u_t}_{L^4}\aiminnorm{\nabla\theta}^{1/2}\aiminnorm{\Delta\theta}^{1/2}\aiminnorm{\theta_t} +  C\aiminnorm{\boldsymbol u_t}_{L^4}\aiminnorm{S}_{H^2}\aiminnorm{\theta_t}.
\]
Using Young's inequality, we derive from \eqref{eq2.2.10} that
\[
	\frac{\mathrm{d}}{\mathrm{d}t} \aiminnorm{\theta_t}^2 + 2\aiminnorm{\nabla\theta_t}^2
	\leq C\aiminnorm{\boldsymbol u_t}_{L^4}^4 + C\aiminnorm{\nabla\theta}^{2}\aiminnorm{\Delta\theta}^2
	+C\aiminnorm{S}_{H^2}^4 + \aiminnorm{\theta_t}^2.
\]
Thus, the Gronwall lemma implies
\[
\sup_{t\in[0, t_1]}\aiminnorm{\theta_t(t)}^2 + 2\int_0^{t_1} \aiminnorm{\nabla\theta_t(t)}^2 {\mathrm{d}t}
\leq e^{t_1}\big( \aiminnorm{\theta_t(0)}^2 + C\int_0^{t_1}\big[\aiminnorm{\boldsymbol u_t}_{L^4}^4 + \aiminnorm{\nabla\theta}^{2}\aiminnorm{\Delta\theta}^2+\aiminnorm{S}_{H^2}^4 \big]{\mathrm{d}t} \big),
\]
and from the equation \eqref{eq1.1.3}$_3$, one has
\[
\aiminnorm{\theta_t(0)}\leq \aiminnorm{\Delta\theta_0} + \aiminnorm{\boldsymbol u_0}_{L^4}\aiminnorm{\nabla\theta_0}_{L^4} + \aiminnorm{\boldsymbol u_0}_{L^4}\aiminnorm{\nabla S}_{L^4} + \aiminnorm{\Delta S}.
\]
Therefore, together with \eqref{eq2.2.8}, we find
\begin{equation}\label{eq2.2.12}
\sup_{t\in[0, t_1]}\aiminnorm{\theta_t(t)}^2 + 2\int_0^{t_1} \aiminnorm{\nabla\theta_t(t)}^2 {\mathrm{d}t}
\leq \mathcal Q_3(t_1, \aiminnorm{\omega_0}_{L^4}, \aiminnorm{\theta_0}_{H^2}, \aiminnorm{S}_{H^2}).
\end{equation}
From the equation \eqref{eq1.1.3}$_3$ again, we obtain
\[
\aiminnorm{\Delta\theta}\leq \aiminnorm{\theta_t} + \aiminnorm{\boldsymbol u}_{L^4}\aiminnorm{\nabla\theta}_{L^4} + \aiminnorm{\boldsymbol u}_{L^4}\aiminnorm{\nabla S}_{L^4} + \aiminnorm{\Delta S},
\]
which, together with the estimates \eqref{eq2.2.8} and \eqref{eq2.2.12}, immediately gives
\begin{equation}\label{eq2.2.13}
	\sup_{t\in[0, t_1]}\aiminnorm{\Delta\theta(t)} \leq \mathcal Q_3(t_1, \aiminnorm{\omega_0}_{L^4}, \aiminnorm{\theta_0}_{H^2}, \aiminnorm{S}_{H^2}).
\end{equation}
Taking the gradient $\nabla$ on \eqref{eq1.1.3}$_3$ and we similarly have
\begin{equation}\begin{split}\label{eq2.2.14}
\aiminnorm{\nabla\Delta\theta}
&\leq \aiminnorm{\nabla\theta_t} + \aiminnorm{\nabla(\boldsymbol u\cdot\nabla\theta)} + \aiminnorm{\nabla(\boldsymbol u\cdot\nabla S)} + \aiminnorm{\nabla\Delta S}\\
&\leq \aiminnorm{\nabla\theta_t} + \aiminnorm{\nabla\boldsymbol u}_{L^4}\aiminnorm{\nabla\theta}_{L^4} + \aiminnorm{\boldsymbol u}_{L^4}\aiminnorm{\Delta\theta}_{L^4} \\
&\quad+ \aiminnorm{\nabla\boldsymbol u}_{L^4}\aiminnorm{\nabla S}_{L^4} +\aiminnorm{\boldsymbol u}_{L^4}\aiminnorm{\Delta S}_{L^4} + \aiminnorm{S}_{H^3}.
\end{split}\end{equation}
The trouble term in \eqref{eq2.2.14} is $\aiminnorm{\boldsymbol u}_{L^4}\aiminnorm{\Delta\theta}_{L^4}$, which can be estimated by Ladyzhenskaya's and Young's inequalities:
\[
\aiminnorm{\boldsymbol u}_{L^4}\aiminnorm{\Delta\theta}_{L^4} \leq C\aiminnorm{\boldsymbol u}_{L^4}\aiminnorm{\Delta\theta}^{1/2}\aiminnorm{\nabla\Delta\theta}^{1/2} \leq C\aiminnorm{\boldsymbol u}_{L^4}^2\aiminnorm{\Delta\theta} + \frac14\aiminnorm{\nabla\Delta\theta}.
\]
Hence, by the Sobolev embedding, the inequality \eqref{eq2.2.14} becomes
\begin{equation*}\begin{split}
	\aiminnorm{\nabla\Delta\theta}
	\leq C\big(&\aiminnorm{\nabla\theta_t} +
	\aiminnorm{\nabla\boldsymbol u}_{L^4}\aiminnorm{\theta}_{H^1} + \aiminnorm{\boldsymbol u}_{L^4}^2\aiminnorm{\Delta\theta}\\
&+ \aiminnorm{\nabla\boldsymbol u}_{L^4}\aiminnorm{ S}_{H^2} +\aiminnorm{\boldsymbol u}_{L^4}\aiminnorm{ S}_{H^2} + \aiminnorm{S}_{H^3}\big).
\end{split}\end{equation*}
which, by utilizing the estimates \eqref{eq2.2.8} and \eqref{eq2.2.12}-\eqref{eq2.2.13}, shows
\begin{equation}\label{eq2.2.16}
	\int_0^{t_1}\aiminnorm{\nabla\Delta\theta(t)}^2{\mathrm{d}s}\leq  \mathcal Q_3(t_1, \aiminnorm{\omega_0}_{L^4}, \aiminnorm{\theta_0}_{H^2}, \aiminnorm{S}_{H^3}).
\end{equation}
We thus proved the first two estimates in \eqref{eq2.0.6} and we now turn to the $L^\infty$-estimate of the vorticity $\omega$.

By the Sobolev embedding, we  have
\begin{equation}\label{eq2.2.18}
\aiminnorm{\theta\boldsymbol e_2 + S e_2}_{L^2(0, t_1; W^{1,\infty}(\Omega))}\leq \mathcal Q_3(t_1, \aiminnorm{\omega_0}_{L^4}, \aiminnorm{\theta_0}_{H^2}, \aiminnorm{S}_{H^3}).
\end{equation}
At this point, applying Proposition \eqref{eqa.1} again, we read from \eqref{eqa.3} that
\begin{equation*}\begin{split}
\aiminnorm{\omega}_{L^\infty(0, t_1; L^p(\Omega))}
\leq \aiminnorm{{\mathrm{curl}\,}\boldsymbol u_0}_{L^p} + \int_0^{t_1}\aiminnorm{\theta + S}_{W^{1,p}}{\mathrm{d}s},\quad p\geq 2,
\end{split}\end{equation*}
and letting $p\rightarrow\infty$ and using \eqref{eq2.2.18} yield the last estimate in \eqref{eq2.0.6}. This completes the proof of regularity part of Theorem~\ref{thm2.1}.

\subsection{Uniqueness}
Let $(\boldsymbol u_1,  \theta_1)$ and $(\boldsymbol u_2, \theta_2)$ be two solutions of the 2D Euler-Boussinesq system \eqref{eq1.1.3}-\eqref{eq1.1.4} satisfying \eqref{eq2.0.6}. Observe that from \eqref{eq1.5.5} and \eqref{eq2.0.6},
\begin{equation}\label{eq2.2.24}
	\aiminnorm{\nabla \boldsymbol u_j(t)}_{L^{\gamma_{\mathrm{exp}}}} \leq C_\Omega  \aiminnorm{\boldsymbol \omega_j}_{L^{\infty}}
	\leq \mathcal Q_4 ,\quad \omega_j={\mathrm{curl}\,} \boldsymbol u_j,\qquad\forall\,t\in[0, t_1].
\end{equation}
The differences $\boldsymbol u=\boldsymbol u_2 - \boldsymbol u_1$ and $\theta=\theta_2 - \theta_1$ then satisfy the equations
\begin{equation}\begin{cases}\label{eq2.2.25}
	\partial_t\boldsymbol u  + \boldsymbol u_2\cdot \nabla \boldsymbol u +\boldsymbol u\cdot\nabla\boldsymbol u_1+ \nabla\pi = \theta \boldsymbol e_2,\quad \boldsymbol e_2=(0,1),\\
	\partial_t \theta - \Delta\theta + \boldsymbol u_2\cdot\nabla\theta + \boldsymbol u\cdot\nabla\theta_1 +\boldsymbol u\cdot\nabla S = 0, \\
\end{cases}\end{equation}
for some pressure function $\pi=\pi(t,x,y)$.

Taking  the inner product \eqref{eq2.2.25}$_1$ with $\boldsymbol u$ in $L^2(\Omega)$, (legitimately) integrating by parts, and applying the Orlicz space version of H\"older's inequality (see Proposition~\ref{prop1.1.2}), we arrive at
\begin{equation*}\begin{split}
	\frac12\frac{\mathrm{d}}{\mathrm{d}t}\aiminnorm{\boldsymbol u}^2
	&\leq \aiminnorm{\theta}\aiminnorm{\boldsymbol u} + \aiminnorm{\nabla\boldsymbol u_1}_{L^{\gamma_{\mathrm{exp}}}}\aiminnorm{|\boldsymbol u|^2}_{L^{\gamma_{\mathrm{exp}}^*}}, \\
	&\leq \frac14\aiminnorm{\Delta\theta}^2 + C\aiminnorm{\boldsymbol u}^2 + \mathcal Q_4 \epsilon^{-1/(1+\epsilon)}\aiminnorm{ |\boldsymbol u|^2 }_{L^{1+\epsilon}},\qquad \forall\,0<\epsilon\leq 1.
\end{split}\end{equation*}
where we used the inequalities \eqref{eq1.5.3} and \eqref{eq2.2.24} and the Poincar\'e inequality for $\theta$. Furthermore, by the interpolation inequality,
\begin{equation}\begin{split}\label{eq2.2.27}
\frac12\frac{\mathrm{d}}{\mathrm{d}t}\aiminnorm{\boldsymbol u}^2
	&\leq \frac14\aiminnorm{\Delta\theta}^2 + C\aiminnorm{\boldsymbol u}^2 + \mathcal Q_4 \epsilon^{-1/(1+\epsilon)}\aiminnorm{ \boldsymbol u }^{2/(1+\epsilon)}\aiminnorm{\boldsymbol u}_{L^\infty}^{\epsilon/(1+\epsilon)},\qquad \forall\,0<\epsilon\leq 1.	
\end{split}\end{equation}

Multiplying \eqref{eq2.2.25}$_2$ by $-\Delta\theta$ and integrating in $\Omega$, we deduce from  Ladyzhenskaya's and Young's inequalities and the Sobolev embedding that
\begin{equation}\begin{split}\label{eq2.2.28}
	\frac12\frac{\mathrm{d}}{\mathrm{d}t}\aiminnorm{\nabla\theta}^2+\aiminnorm{\Delta\theta}^2
	&\leq C\aiminnorm{\boldsymbol u_2}_{L^4}\aiminnorm{\nabla\theta}^{1/2}\aiminnorm{\Delta\theta}^{3/2} + (\aiminnorm{\nabla\theta_1}_{L^\infty}+\aiminnorm{\nabla S}_{L^\infty})\aiminnorm{\boldsymbol u}\aiminnorm{\Delta\theta} \\
	&\leq C\aiminnorm{\boldsymbol u_2}_{L^4}^4\aiminnorm{\nabla\theta}^2 + C(\aiminnorm{\theta_1}_{H^3}^2 + \aiminnorm{S}_{H^3}^2) \aiminnorm{\boldsymbol u}^2 + \frac14\aiminnorm{\Delta\theta}^2.
\end{split}\end{equation}
Adding the inequalities \eqref{eq2.2.27}-\eqref{eq2.2.28} together gives
\begin{equation}\begin{split}\label{eq2.2.30}
	\frac{\mathrm{d}}{\mathrm{d}t}( \aiminnorm{\boldsymbol u}^2 + \aiminnorm{\nabla\theta}^2 )
	&\leq C\aiminnorm{\boldsymbol u}^2 + \mathcal Q_4 \epsilon^{-1/(1+\epsilon)}\aiminnorm{ \boldsymbol u }^{2/(1+\epsilon)}\aiminnorm{\boldsymbol u}_{L^\infty}^{\epsilon/(1+\epsilon)}\\
	&\quad+C\aiminnorm{\boldsymbol u_2}_{L^4}^4\aiminnorm{\nabla\theta}^2 + C(\aiminnorm{\theta_1}_{H^3}^2 + \aiminnorm{S}_{H^3}^2) \aiminnorm{\boldsymbol u}^2\\
	&\leq \kappa_1 \epsilon^{-1/(1+\epsilon)}( \aiminnorm{\boldsymbol u}^2 + \aiminnorm{\nabla\theta}^2)^{1/(1+\epsilon)}
	+ g(t)( \aiminnorm{\boldsymbol u}^2 + \aiminnorm{\nabla\theta}^2),
\end{split}\end{equation}
where
\begin{equation*}\begin{split}
\kappa_1:&=\mathcal Q_4\aiminnorm{\boldsymbol u}_{L^\infty(\Omega\times(0, t_1))}^{\epsilon/(1+\epsilon)} < \infty,\qquad\text{ independent of }\epsilon,\\
g(t):&=C(1+\aiminnorm{\boldsymbol u_2}_{L^4}^4 +\aiminnorm{\theta_1}_{H^3}^2 + \aiminnorm{S}_{H^3}^2) \in L^1(0, t_1),
\end{split}\end{equation*}
which stems from the estimate \eqref{eq2.0.6}.

We denote by $Y(t):=\aiminnorm{\boldsymbol u(t)}^2 + \aiminnorm{\nabla\theta(t)}^2$,
then the differential inequality \eqref{eq2.2.30} turns into
\begin{equation*}\begin{split}
		\frac{\mathrm{d}}{\mathrm{d}t} Y(t)
		\leq \kappa_1 \epsilon^{-1/(1+\epsilon)} Y(t)^{ 1/(1+\epsilon)} +  g(t)Y(t),
\end{split}\end{equation*}
which, by letting $\widetilde Y(t) = e^{-\int_0^t g(s){\mathrm{d}s}}Y(t)$, implies
\begin{equation}\label{eq2.2.32}
	\frac{\mathrm{d}}{\mathrm{d}t} \widetilde Y(t) \leq \kappa_1 \epsilon^{-1/(1+\epsilon)}  e^{-\int_0^t g(s){\mathrm{d}s}} Y(t)^{ 1/(1+\epsilon)}  \leq \kappa_1 \epsilon^{-1/(1+\epsilon)} \widetilde Y(t)^{ 1/(1+\epsilon)}, \quad \forall\,0<\epsilon\leq 1.
\end{equation}
We compute
\[
\frac{\mathrm{d}}{\mathrm{d}t}\widetilde Y(t) = \frac{\mathrm{d}}{\mathrm{d}t}\bigg(\widetilde Y(t)^{\epsilon/(1+\epsilon)}\bigg)^{(1+\epsilon)/\epsilon}
=\frac{1+\epsilon}{\epsilon} \widetilde Y(t)^{1/(1+\epsilon)} \frac{\mathrm{d}}{\mathrm{d}t} \big(\widetilde Y(t)^{\epsilon/(1+\epsilon)}\big),
\]
and deduce from \eqref{eq2.2.32} that
\begin{equation}\label{eq2.2.33}
\frac{\mathrm{d}}{\mathrm{d}t} \big(\widetilde Y(t)^{\epsilon/(1+\epsilon)}\big)
\leq \epsilon^{\epsilon/(1+\epsilon)} \frac{\kappa_1}{1+\epsilon},\qquad \forall\,0<\epsilon\leq 1.
\end{equation}
Due to the continuity $(\boldsymbol u, \theta)\in \mathcal C([0, t_1]; L_{\boldsymbol\tau}^2(\Omega))\times\mathcal C([0, t_1]; H_0^1(\Omega))$, the functional $\widetilde Y$ is continuous on $[0, t_1]$. Noting that $\widetilde Y(0)=0$ and integrating \eqref{eq2.2.33} in time $(0, t)$ gives
\begin{equation}\label{eq2.2.34}
	\widetilde Y(t)
	\leq \epsilon \bigg(\frac{\kappa_1 t}{1+\epsilon}\bigg)^{(1+\epsilon)/\epsilon},\qquad \forall t>0, \qquad \forall\,0<\epsilon\leq 1.
\end{equation}
Choosing $t^*>0$ small enough such that $\kappa_1 t^*/(1+\epsilon)\leq \kappa_1 t^*/2 < 1$ and letting $\epsilon$ tend to $0$ in \eqref{eq2.2.34} entails that $\widetilde Y(t) \equiv 0$ on $[0, t^*]$. By the induction method, we can conclude that $\widetilde Y(t) \equiv 0$ and hence $Y(t)\equiv 0$ on $[0, t_1]$. This completes the proof of uniqueness part of Theorem~\ref{thm2.1}.

\begin{remark}
We note that the proof of uniqueness shows the continuity of the solution semigroup in the topology of $L_{\boldsymbol\tau}^2(\Omega)\times H_0^1(\Omega)$ within the set $\aiminset{\boldsymbol u\in L_{\boldsymbol\tau}^2(\Omega)\,:\,{\mathrm{curl}\,}\boldsymbol u\in L^\infty(\Omega)}\times H_0^1(\Omega)\cap H^2(\Omega)$. It is not clear whether the semigroup is continuous on $L^p_{\boldsymbol\tau}(\Omega)\times H^2(\Omega)$ for some $p>2$.	
\end{remark}

\appendix
\section{A preliminary result for the 2D Euler equations}\label{sec-app-euler}
In the appendix, we consider the standard 2D Euler equations in the polygonal-like domain $\Omega$ which read
\begin{equation}\begin{cases}\label{eqa.1}
\partial_t\boldsymbol u  + \boldsymbol u\cdot \nabla \boldsymbol u + \nabla\pi = \boldsymbol f,\\
	{\mathrm{div}\,}\boldsymbol u=0,\\	
\end{cases}\end{equation}
with the initial and boundary conditions
\begin{equation}
	\boldsymbol u(0)=\boldsymbol u_0,\text{ in }\Omega,\qquad \boldsymbol u\cdot\boldsymbol n=0,\text{ on }\partial\Omega.
\end{equation}
The 2D Euler equations can be expressed in terms of vorticity which allows to underline the conservation of the vorticity. This fact will turn out to be of great relevance in the analysis of the two-dimensional flow as we will see below. The vorticity formulation of the Euler equations also has the advantage of having eliminated the pressure term $\nabla\pi$ and it reads
\begin{equation}\begin{cases}\label{eqa.2}
	\partial_t \omega  + \boldsymbol u\cdot \nabla \omega  = {\mathrm{curl}\,}\boldsymbol f,\\
	{\mathrm{div}\,}\boldsymbol u=0,\qquad \boldsymbol u=\boldsymbol{\mathrm{K}}_\Omega\omega=\nabla^{\perp}\mathrm{G}_{\Omega} \omega,\\
	\omega(0)={\mathrm{curl}\,}\boldsymbol u_0,
\end{cases}\end{equation}
where $\omega={\mathrm{curl}\,}\boldsymbol u=\partial_x u_2 - \partial_y u_1$.

We recall the following standard $L^p$ a priori estimates for the 2D Euler equations (see for example \cite[Lemmas~4.3, 4.5]{BDT13}).
\begin{proposition}\label{propa.1}
	Let $\Omega$ be a polygonal-like domain with maximum aperture $\max\alpha_j\leq \pi/2$ and let $2\leq p<\infty$ and define (see \cite[Remark 4.1]{BDT13})
	\[
	s(p)=\begin{cases}
		\text{any }s\in[1,2), &p=2,\\
		p, &p>2.
	\end{cases}
	\]
	Assume that $\omega_0={\mathrm{curl}\,}\boldsymbol u_0$ belongs to $L^{p}(\Omega)$ and $\boldsymbol f$ belongs to $L^2(0, t_1; W^{1,p}(\Omega))$. If $\boldsymbol u$ is a solution of \eqref{eqa.1}, then
	\begin{equation}\label{eqa.3}
	\omega={\mathrm{curl}\,}\boldsymbol u\in L^\infty(0, t_1; L^p(\Omega)), \quad \aiminnorm{\omega}_{L^\infty(0, t_1; L^p(\Omega) )}\leq \aiminnorm{{\mathrm{curl}\,}\boldsymbol u_0}_{L^p} + \int_0^{t_1}\aiminnorm{{\mathrm{curl}\,}\boldsymbol f}_{L^p}{\mathrm{d}s},
	\end{equation}
	and for all $2\leq p<\infty$, there holds
	\begin{equation}\begin{split}\label{eqa.4}
		\boldsymbol u\in L^\infty(0, t_1; W^{1,p}(\Omega)),\qquad \partial_t\boldsymbol u\in L^2(0, t_1; L^{s(p)}(\Omega)),\\
		\aiminnorm{\boldsymbol u}_{ L^\infty(0, t_1; W^{1,p}(\Omega)) } + \aiminnorm{\partial_t \boldsymbol u}_{ L^2(0, t_1; L^{s(p)}(\Omega)) }\leq \mathcal Q_8,
	\end{split}\end{equation}
	where $\mathcal Q_8$ is a positive function defined by
	\begin{equation}\label{eqa.5}
		\mathcal Q_8:=\mathcal Q_8(p, \aiminnorm{\omega_0}_{L^{p}}, \aiminnorm{\boldsymbol f}_{L^2(0, t_1; W^{1,p}(\Omega))}),
	\end{equation}
	which is increasing on its arguments.
	
	Furthermore, if $\boldsymbol f\in L^q(0, t_1; W^{1,p}(\Omega))$ for some $q\geq 2$, then we actually have
	\begin{equation}\label{eqa.6}
	\partial_t\boldsymbol u\in L^q(0, t_1; L^{s(p)}(\Omega)),\quad
	\aiminnorm{\partial_t \boldsymbol u}_{ L^q(0, t_1; L^{s(p)}(\Omega)) }\leq  \mathcal Q_9(p, \aiminnorm{\omega_0}_{L^{p}}, \aiminnorm{\boldsymbol f}_{L^q(0, t_1; W^{1,p}(\Omega))}).	
	\end{equation}
\end{proposition}
We present the proof of Proposition \ref{propa.1} for the sake of completeness.
\begin{proof}[Proof of Proposition \ref{propa.1}]
Multiplying \eqref{eqa.2} by $p\aiminabs{\omega}^{p-2}\omega$, integrating on $\Omega$, and using the fact
\[
p\aimininner{\boldsymbol u\cdot\nabla \omega}{ \aiminabs{\omega}^{p-2}\omega } =\int_\Omega \boldsymbol u\cdot \nabla \aiminabs{\omega}^p{\mathrm{d}x}{\mathrm{d}y}= 0,
\]
which stems from that $\boldsymbol u$ is divergence free and has zero normal component on $\partial\Omega$,
 we obtain
\[
\frac{\mathrm{d}}{\mathrm{d}t}\aiminnorm{\omega}_{L^p}^p = p\aimininner{{\mathrm{curl}\,}\boldsymbol f}{\aiminabs{\omega}^{p-2}\omega}
\leq p\aiminnorm{{\mathrm{curl}\,} \boldsymbol f}_{L^p} \aiminnorm{\omega}_{L^p}^{p-1},
\]
which implies
\[
\frac{\mathrm{d}}{\mathrm{d}t}\aiminnorm{\omega}_{L^p} \leq \aiminnorm{{\mathrm{curl}\,} \boldsymbol f}_{L^p},
\]
Integrating in time $(0, t)$ and then taking the sup over $[0,t_1]$ show the desired estimate \eqref{eqa.3}.
Now, the Biot-Savart law \eqref{eq1.3.4} yields
\begin{equation}\label{eqa.10}
	\boldsymbol u=\boldsymbol{\mathrm{K}}_\Omega \omega\in L^\infty(0, t_1; W^{1,p}(\Omega)\cap L_{\boldsymbol\tau}^{2}(\Omega)),\qquad
	\aiminnorm{\boldsymbol u}_{L^\infty(0, t_1; W^{1,p}(\Omega))}\leq \mathcal Q_8.
\end{equation}
where $\mathcal Q_8$ is defined by \eqref{eqa.5}.

Applying $\mathrm{P}_\Omega$ on \eqref{eqa.1} gives the identity
\begin{equation}\label{eqa.12}
\partial_t\boldsymbol u=\mathrm{P}_\Omega( - \boldsymbol u\cdot\nabla\boldsymbol u + \boldsymbol f).
\end{equation}
First, by the  Sobolev embedding and \eqref{eqa.10}, we obtain
\begin{equation*}
	\aiminnorm{\boldsymbol u}_{L^\infty(0, t_1; L^r(\Omega))}
	\lesssim \aiminnorm{\boldsymbol u}_{L^\infty(0, t_1; W^{1,p}(\Omega))}
	\leq \mathcal Q_8,\qquad r=\begin{cases}
		s(p)^*, &p=2,\\
		\infty, &p>2,\\
	\end{cases}
\end{equation*}
where  $s(p)^*$ are the Sobolev conjugate exponent $s(p)$, that is
\[
s(p)^*=\frac{2s(p)}{2-s(p)},\qquad 1\leq s(p) <2,\\
\]
and we then deduce from H\"older's inequality that
\begin{equation}\label{eqa.13}
	\aiminnorm{\boldsymbol u\cdot\nabla\boldsymbol u}_{L^\infty(0, t_1; L^{s(p))}(\Omega)} \leq \aiminnorm{\boldsymbol u}_{L^\infty(0, t_1; L^r(\Omega))}\aiminnorm{\nabla\boldsymbol u}_{L^\infty(0, t_1; L^p(\Omega))}
	\leq \mathcal Q_8.
\end{equation}
Proposition~\ref{prop1.1.1} guarantees that $\mathrm{P}_\Omega$ is a linear bounded operator on $L^{s(p)}$ and we find from \eqref{eqa.12}-\eqref{eqa.13} that
\begin{equation*}
\begin{split}
	\aiminnorm{\partial_t\boldsymbol u}_{L^q(0, t_1; L^{s(p)}(\Omega))}
	&\leq \aiminnorm{\boldsymbol u\cdot\nabla\boldsymbol u}_{L^q(0, t_1; L^{s(p)}(\Omega))} + \aiminnorm{\boldsymbol f}_{L^q(0, t_1; L^{s(p)}(\Omega))}\\
	&\leq \max(t_1, 1)\aiminnorm{\boldsymbol u\cdot\nabla\boldsymbol u}_{L^\infty(0, t_1; L^{s(p)}(\Omega))} + \aiminnorm{\boldsymbol f}_{L^q(0, t_1; L^{s(p)}(\Omega))},
\end{split}
\end{equation*}
which implies \eqref{eqa.4} by the estimate \eqref{eqa.10} and letting $q=2$ and also implies \eqref{eqa.6} if $\boldsymbol f$ belongs to $L^q(0, t_1; W^{1,p}(\Omega))$. This ends the proof of Proposition~\ref{propa.1}.
\end{proof}

\bibliographystyle{amsalpha}
\providecommand{\bysame}{\leavevmode\hbox to3em{\hrulefill}\thinspace}
\providecommand{\MR}{\relax\ifhmode\unskip\space\fi MR }
\providecommand{\MRhref}[2]{%
  \href{http://www.ams.org/mathscinet-getitem?mr=#1}{#2}
}
\providecommand{\href}[2]{#2}

\end{document}